\documentclass[a4paper,12pt]{article}

\usepackage[ascii]{inputenc}
\usepackage{amsmath,amsfonts,amssymb,amsthm}
\usepackage{mathrsfs}
\usepackage{comment}

\newtheoremstyle{theorem}{5pt}{5pt}{\itshape}{}{\bfseries}{.}{.5em}{}

\theoremstyle{theorem}
\newtheorem{theorem}{Theorem}
\newtheorem{lemma}[theorem]{Lemma}
\newtheorem{corollary}[theorem]{Corollary}

\usepackage{titlesec}
\titleformat{\section}
{\Large\bfseries}{\thesection}{1em}{\large}
\titlespacing*{\section}{0pt}{3.5ex plus 1ex minus .2ex}{2.3ex plus .2ex}

\titleformat{\subsection}
{\normalfont\bfseries}{\thesubsection}{1em}{\normalsize}
\titlespacing*{\section}{0pt}{3.5ex plus 1ex minus .2ex}{2.3ex plus .2ex}

\allowdisplaybreaks[3]

\begin{document}

\title{Resonances and $\Omega$-results for Exponential \\ Sums Related to Maass Forms for $\mathrm{SL}(n,\mathbb Z)$}
\author{Anne-Maria Ernvall-Hyt\"onen\footnote{Email: anne-maria.ernvall-hytonen@helsinki.fi},
Jesse J\"a\"asaari\footnote{Email: jesse.jaasaari@helsinki.fi},
Esa V\!. Vesalainen\footnote{Correspondence by email to esa.vesalainen@helsinki.fi, by phone to +358 (0) 44 562 5504, or by mail to Esa V\!. Vesalainen, Department of Mathematics and Statistics, P.O. Box 68, FI-00014 University of Helsinki, FINLAND.}\\
\normalsize Department of Mathematics and Statistics, University of Helsinki}
\date{}
\maketitle

\begin{abstract}
\noindent
We obtain resonances for short exponential sums weighted by Fourier coefficients of Maass forms for $\mathrm{SL}(n,\mathbb Z)$. This involves deriving asymptotics for the integrals appearing in the $\mathrm{GL}(n)$ Voronoi summation formula. As an application, we also prove an $\Omega$-result for short sums of Fourier coefficients.
\end{abstract}

\section{Introduction}

\subsection{Exponential sums related to cusp forms}

Little is known about Fourier coefficients of holomorphic cusp forms or Maass forms both in $\mathrm{GL}(2)$ and in $\mathrm{GL}(n)$ for $n>2$. Therefore, it makes sense to study exponential sums
\begin{equation}\label{linear}
\sum_{M\leqslant m\leqslant M+\Delta}a(m)\,e(m\alpha)
\end{equation}
of the Fourier coefficients of cusp forms with $M\in\left[1,\infty\right[$ and $\Delta\in\left[1,M\right]$. Instead of the function above, one might want to have some other function in the place of the function $\alpha m$, say, the function $\alpha m^{\beta}$. However, in what follows, we concentrate on so called linear exponential sums, that is, sums of the type \eqref{linear} with $\alpha\in \mathbb{R}$. Furthermore, we may assume that $0\leqslant\alpha<1$.

In particular, short sums are interesting for studying properties of Fourier coefficients, and they are a natural analogue for studying classical number-theoretic error terms in short intervals.

The behaviour of the exponential sums is extremely intriguing: One might expect square-root cancellation, and while this is not an unreasonable assumption in the GL(2) setting when $\Delta \ll M^{1/2}$, Ernvall-Hyt\"onen \cite{Ernvall-Hytonen--Beograd} has proved that
\begin{align}\label{reso}
&\sum_{M\leqslant m\leqslant M+\Delta}a(m)\,e\!\left(\frac{m\,\sqrt d}{\sqrt M}\right)w(m)\notag \\
&\qquad=C\,a(d)\,d^{-1/4}\int\limits_M^{M+\Delta}x^{-1/4}\,w(x)\,e\!\left(\frac{x\,\sqrt d}{\sqrt M}-2\sqrt{dx}\right)\mathrm dx+O(1),
\end{align}
for the normalised Fourier coefficients $a(n)$ of a holomorphic cusp form
\[\sum_{n=1}^\infty a(n)\,n^{(\kappa-1)/2}\,e(nz)\] of weight $\kappa\in\mathbb Z_+$ for $\mathrm{SL}(2,\mathbb Z)$, where $d\in\mathbb Z_+$ is fixed and $C$ is a non-zero constant only depending on the weight of the underlying holomorphic cusp form, and $w$ is a suitable weight function supported on $\left[M,M+\Delta\right]$. Since the integrand  does not oscillate, the integral is $\asymp\Delta\,M^{-1/4}$. This shows that for $M^{1/2+\varepsilon}\ll \Delta\leqslant M^{3/4}\,d^{-1/2}$ the size of the sum is $\asymp \Delta\,M^{-1/4}$ unless $a(d)=0$, in which case the size of the sum is $O(1)$. The estimate $\Delta\,M^{-1/4}$ is not only larger than square root size, but also grows linearly with $\Delta$, which is quite surprising. A similar phenomenon has been observed for $\mathrm{GL}(2)$ and $\mathrm{GL}(3)$ Maass forms by Ernvall-Hyt\"onen \cite{Ernvall-Hytonen, Ernvall-Hytonen-siauliai}, and for non-linear sums by Iwaniec, Luo and Sarnak \cite{Iwaniec--Luo--Sarnak} and
 by Ren and Ye \cite{Ren-Ye}.

$\Omega$-results are the other side of the story. With an $\Omega$-result we mean a lower bound in the following sense:
\[
\sum_{M\leqslant m\leqslant M+\Delta}a(m)=\Omega(F(M,\Delta))
\]
means that
\[
\sum_{M\leqslant m\leqslant M+\Delta}a(m)\neq o(F(M,\Delta)).
\]

The result above \eqref{reso} has been used to prove the following $\Omega$-result:
\[\sum_{M\leqslant m\leqslant M+cM^{1/2}}a(m)=\Omega(M^{1/4}),\]
where $c\in\mathbb R_+$ is an arbitrary fixed coefficient.
This $\Omega$-result extends the earlier work  by Ivi\'c \cite{Ivic2009} (see also \cite{Jutila1984}), where he has shown that
\[
\sum_{M\leqslant m\leqslant M+\Delta}a(m)=\Omega(\sqrt{\Delta})
\]
for $\Delta=o(\sqrt{M})$.

While upper and lower bounds are relatively widely studied in the $\mathrm{GL}(2)$ setting (see e.g. \cite{Wilton, Jutila1987b, Epstein--Hafner--Sarnak1985, Hafner1987, Hafner--Ivic}), even though there are plenty of open questions left, the situation in $\mathrm{GL}(n)$ for $n>2$ is much less understood. The best uniform upper bound for exponential sums for $\mathrm{GL}(3)$ is the one due to Miller \cite{Miller},
\[\sum_{m\leqslant M}A(m,1)\,e(m\alpha)\ll M^{3/4+\varepsilon},\]
which holds uniformly in $\alpha\in\mathbb R$.
Ren and Ye \cite{Ren-Ye--gl3} studied how the bound can be improved in the presence of a weight function in terms of rational approximations of $\alpha$. The dependence on weight and/or spectral parameters has been studied in both $\mathrm{GL}(2)$ and in $\mathrm{GL}(3)$ \cite{Li--Young2012, Godber2013}.

\subsection{What we do in this paper}

Let $\psi$ be a Maass form for $\mathrm{SL}(n,\mathbb Z)$, where $n\geqslant2$ (for a general reference, see \cite{Goldfeld}). It comes with an attached $L$-function called the Godement--Jacquet $L$-function:
\[L(s)=\sum_{m=1}^\infty\frac{A(m,1,\ldots,1)}{m^s}.\]
The Fourier coefficients $A(m,1,\ldots,1)$ satisfy a pointwise bound
\[A(m,1,\ldots,1)\ll m^{\vartheta+\varepsilon}\]
for every $m\in\mathbb Z_+$ for some $\vartheta\in\left[0,\infty\right[$ depending on $n$. It is known that we can choose $\vartheta=0$ for holomorphic cusp forms \cite{Deligne}, $\vartheta=\frac{7}{64}$ for Maass forms in $\mathrm{GL}(2)$, $\vartheta=\frac{5}{14}$ for Maass forms in $\mathrm{GL}(3)$ and $\vartheta=\frac{9}{22}$ for Maass forms in $\mathrm{GL}(4)$ \cite{Kim-Sarnak}. For $\mathrm{GL}(n)$ Maass forms with $n\geq 5$, we can choose $\vartheta=\frac{1}{2}-\frac{2}{n^2+1}$ \cite{Luo-Rudnick-Sarnak}. It has been conjectured that $\vartheta=0$ is admissible. Also, from the study of the Rankin--Selberg convolutions, it follows that
\[\sum_{m\leqslant x}\left|A(m,1,\ldots,1)\right|^2\ll x\]
for $x\in\mathbb R_+$ (see e.g.\ Sect.\ 12.1.\ in \cite{Goldfeld}).

Our object here is to study certain resonances for short exponential sums involving the coefficients $A(m,1,\ldots,1)$. This will lead to an $\Omega$-result for sums of coefficients.

We start by considering the Voronoi type summation formula for exponential sums involving Fourier coefficients of cusp forms. Namely, we derive asymptotics for the integrals in the Voronoi type summation formula for $\mathrm{GL}(n)$ Maass forms.
The details can be found in Section \ref{asymptotics-section}.

We then use these asymptotics to prove the following resonance result:
\begin{theorem}\label{resonance}
Let $M^{1-1/n+\varepsilon}\ll\Delta\ll M$, and let $d$ be a fixed positive integer. Also, let $w\in C_{\mathrm c}^\infty(\mathbb R_+)$ be supported in the interval $[M,M+\Delta]$ with $w^{(\nu)}(x)\ll_\nu\Delta^{-\nu}$ for $\nu\in\mathbb Z_+\cup\left\{0\right\}$. Then
\begin{align*}
&\sum_{M\leqslant m\leqslant M+\Delta}A(m,1,...,1)\,w(m)\,e\!\left(\frac{d^{1/n}\,m}{M^{1-1/n}}\right)
=\frac{A(1,...,1,d)}{d^{1/2-1/(2n)}\,\sqrt n }\,e\!\left(\frac{n+3}8\right)\\
&\qquad\cdot\int\limits_M^{M+\Delta}w(x)\,e\!\left(\frac{d^{1/n}\,x}{M^{1-1/n}}-n\,x^{1/n}\,d^{1/n}\right)x^{1/(2n)-1/2}\,\mathrm dx\\
&\qquad\qquad+O(\Delta\,M^{-1/2-1/(2n)}). 
\end{align*}
\end{theorem}
This is a generalization of Theorem 1.1 from $\cite{Ernvall-Hytonen}$, and a generalization of the corresponding results in the GL(2) setting. Notice that in Theorem 1.1 in \cite{Ernvall-Hytonen}, the constant $\frac{1}{\pi\,d^{-1/3}}$ appears accidentally in front of the main term. It should not be there, and once removed, the result agrees with this theorem.

When $\Delta\ll_{n,d}M^{1-1/(2n)}$, the integrand on the right-hand side in Theorem \ref{resonance} does not oscillate, and so we get the following corollary.
\begin{corollary}\label{weighted-corollary}
Let $M^{1-1/n+\varepsilon}\ll\Delta\ll_{n,d} M^{1-1/(2n)}$, let $d\in\mathbb Z_+$ be fixed, and let $w\in C_{\mathrm c}^\infty(\mathbb R_+)$ be supported in $\left[M,M+\Delta\right]$ and satisfy $w^{(\nu)}(x)\ll_\nu\Delta^{-\nu}$ for every $\nu\in\mathbb Z_+\cup\left\{0\right\}$. Then, if $A(1,\ldots,1,d)\neq0$, we have
\[
\sum_{M\leqslant m\leqslant M+\Delta}A(m,1,\dots ,1)\,e\!\left(\frac{d^{1/n}\,m}{M^{1-1/n}}\right)w(m)\asymp \Delta\,M^{1/(2n)-1/2}.
\]
\end{corollary}
\noindent
In Corollaries \ref{weightless-omega} and \ref{shorter-weightless-omega} below, the weight function is removed. In particular, we get
\begin{corollary}
Let $d\in\mathbb Z_+$ be fixed, and assume that $A(1,\ldots,1,d)\neq0$. Then, for $M^{1-1/n+\varepsilon}\ll\Delta\ll M^{1-1/(2n)}$, the sum
\[\sum_{M\leqslant m\leqslant M+\Delta}A(m,1,\ldots,1)\,e\!\left(\frac{d^{1/n}\,m}{M^{1-1/n}}\right)\]
is $\Omega(\Delta\,M^{1/(2n)-1/2})$, and for $M^{1-1/(2n)}\ll\Delta\ll M$, the sum is $\Omega(M^{1/2})$.
\end{corollary}

One can also prove the following non-linear resonance result:
\begin{theorem}\label{nonlinear-resonance}
Let $M^{1-1/n+\varepsilon}\ll\Delta\ll M$ and let $d$ be a fixed positive integer. Also, let $w\in C_{\mathrm c}^\infty(\mathbb R_+)$ be supported on the interval $[M,M+\Delta]$ with $w^{(\nu)}(x)\ll_\nu\Delta^{-\nu}$ for $\nu\in\mathbb Z_+\cup\left\{0\right\}$. Then
\begin{align*}
&\sum_{M\leqslant m\leqslant M+\Delta}A(m,1,...,1)\,w(m)\,e\!\left(d^{1/n}\,m^{1/n}\right)
=\frac{A(1,...,1,d)}{d^{1/2-1/(2n)}\,\sqrt n}\,e\!\left(\frac{n+3}8\right)\\
&\qquad\cdot\int\limits_M^{M+\Delta}w(x)\,x^{1/(2n)-1/2}\,\mathrm dx
+O(\Delta\,M^{-1/2-1/(2n)}). 
\end{align*}
\end{theorem}
\noindent
We will use Theorem \ref{resonance} to prove the following $\Omega$-result:
\begin{theorem}\label{omega} Assume that $\Delta=o(M^{1/2-\vartheta-\varepsilon})$. Then 
\[
\sum_{M\leqslant m\leqslant M+\Delta}A(m,1,\dots ,1)=\Omega(M^{1/(2n)-1/2}\,\Delta).
\]
\end{theorem}

Finally, it is interesting that one gets the following corollary concerning the determination of all the coefficients $A(m_1,\ldots,m_{n-1})$ from a subset of coefficients.
\begin{corollary}\label{uniqueness-corollary}
Let $M_1$, $M_2$, \dots be a sequence of positive real numbers tending to infinity, and let $\varepsilon$ be an arbitrarily small positive real number. Write
\[I=\mathbb Z\cap\bigcup_{\ell=1}^\infty\left[M_\ell,M_\ell+M_\ell^{1-1/n+\varepsilon}\right].\]
Then the Fourier coefficients $A(m,1,\ldots,1)$ with $m\in I$ uniquely determine all the Fourier coefficients $A(m_1,\ldots,m_{n-1})$.
\end{corollary}
\noindent
Indeed, the coefficients $A(m,1,\ldots,1)$ with $m\in I$ uniquely determine an infinite sequence of exponential sums of the form appearing in Theorem \ref{resonance}, and taking $M\longrightarrow\infty$ we can recover the coefficient $A(1,\ldots,1,d)$ for every $d\in\mathbb Z$. The strong multiplicity one theorem of Jacquet and Shalika \cite{Jacquet--Shalika1981} (Theorem 12.6.1 in \cite{Goldfeld}) implies that the coefficients $A(1,\ldots,1,p)$ for primes $p$ determine all the rest of the coefficients. Corollary \ref{uniqueness-corollary} could perhaps be considered a relative of multiplicity one theorems (see e.g.\ the ones in \cite{Ramakrishnan, Brumley2006}).

\section{Notation}

The symbols $\ll$, $\gg$, $\asymp$, $O$ and $o$ are used for the usual asymptotic notation: for complex valued functions $f$ and $g$ in some set $X$, the notation $f\ll g$ means that $\left|f(x)\right|\leqslant C\left|g(x)\right|$ for all $x\in X$ for some implicit constant $C\in\mathbb R_+$. When the implied constant depends on some parameters $\alpha,\beta,\ldots$, we use $\ll_{\alpha,\beta,\ldots}$ instead of mere $\ll$. The notation $g\gg f$ means $f\ll g$, and $f\asymp g$ means $f\ll g\ll f$. When $f$ and $g$ depend on $M$, we say that $f(M)=o(g(M))$ if $g(M)$ never vanishes and $f(M)/g(M)\longrightarrow0$ as $M\longrightarrow\infty$.

All the implicit constants are allowed to depend on the underlying Maass form (hence also on $n$), on $d$, when it appears, and on $\varepsilon$, which denotes an arbitrarily small fixed positive number, which may not the same on each occurrence. The implicit constants are also allowed to depend on a weight function $w$ when such a function is used.

As usual, complex variables are written in the form $s=\sigma+it$, and we write $e(x)$ for $e^{2\pi ix}$. The subscript in the integral $\int_{(\sigma)}$ means that we integrate over the vertical line $\Re s=\sigma$. For simplicity, we write $\left\langle\cdot\right\rangle$ for $(1+\left|\cdot\right|^2)^{1/2}$.

\section{Basic properties of the Godement--Jacquet $L$-function}

The Godement--Jacquet $L$-function attached to a Maass form $\psi$ for $\mathrm{SL}(n,\mathbb Z)$ with Fourier coefficients $A(m_1,m_2,\ldots,m_{n-1})$ is the Dirichlet series
\[L(s)=\sum_{m=1}^\infty\frac{A(m,1,1,\ldots,1)}{m^s}.\]
This converges absolutely for $\sigma>1$ by the Rankin--Selberg estimate
\[\sum_{m\leqslant x}\left|A(m,1,\ldots,1)\right|^2\ll x.\]
(For this, see e.g. \cite{Goldfeld}, Remark 12.1.8.) The Godement--Jacquet $L$-function has an entire analytic continuation and satisfies the functional equation
\[\pi^{-ns/2}\,G(s)\,L(s)=\pi^{-n(1-s)/2}\,\widetilde G(1-s)\,\widetilde L(1-s),\]
where $\widetilde L$ is the Godement--Jacquet $L$-function of the dual form of $\psi$, and given by the Dirichlet series
\[\widetilde L(s)=\sum_{m=1}^\infty\frac{A(1,1,\ldots,1,m)}{m^s}\]
for $\sigma>1$, and where
\[G(s)=\prod_{\ell=1}^n\Gamma\!\left(\frac{s-\lambda_\ell}2\right)\quad\text{and}\quad
\widetilde G(s)=\prod_{\ell=1}^n\Gamma\!\left(\frac{s-\widetilde\lambda_\ell}2\right),\]
where $\lambda_\ell$ and $\widetilde\lambda_\ell$ are certain complex parameters of $\psi$ and $\widetilde\psi$ with $\sum_{\ell=1}^n\lambda_\ell=\sum_{\ell=1}^n\widetilde\lambda_\ell=0$. It is known that $\Re\lambda_\ell\leqslant\frac12$ and $\Re\widetilde\lambda_\ell\leqslant\frac12$ for each $\ell$.

An elementary application of Stirling's formula says that when $s$ lies in the vertical strips below and has a sufficiently large imaginary part (say, $\left|t\right|\geqslant1$), the multiple $\Gamma$-factors can be replaced by a single quotient of two $\Gamma$-factors:
\[\frac{\widetilde G(1-s)}{G(s)}
=n^{ns-n/2}\frac{\Gamma\!\left(\frac{1-ns}2\right)}{\Gamma\!\left(\frac{ns-(n-1)}2\right)}\left(1+O(\left|s\right|^{-1})\right),\]
as will be seen in Section \ref{asymptotics-section}.

The $L$-function has the usual growth properties in vertical strips. Since the $L$-function is bounded on the line $\sigma=1+\delta$, the functional equation and Stirling's formula (see Theorem \ref{stirling} below) imply that $L(s)\ll\left\langle t\right\rangle^{n/2+n\delta}$ on the line $\sigma=-\delta$. Thus, by the Phr\'agmen--Lindel\"of principle for vertical strips,
\[L(s)\ll\left\langle t\right\rangle^{(1+\delta-\sigma)n/2}\]
in the vertical strip $-\delta\leqslant\sigma\leqslant1+\delta$.

\section{Useful results}

A Voronoi summation formula was implemented for $\mathrm{GL}(n)$ in \cite{Miller--Schmid1, Miller--Schmid2, Goldfeld--Li, Goldfeld--Li2}. The full formula involving additive twists is complicated but we only need the twistless formula, which is the following.
\begin{theorem}\label{gln-voronoi}
Let $f\in C_{\mathrm c}^\infty(\mathbb R_+)$. Then
\begin{align*}
&\sum_{m=1}^\infty A(m,1,\ldots,1)\,f(m)\\
&\qquad=\sum_{m=1}^\infty\frac{A(1,\ldots,1,m)}m\,\frac1{2\pi i}\int\limits_{(-\sigma_0)}
\widetilde f(s)\,\pi^{-n/2}\,\frac{\widetilde G(1-s)}{G(s)}\left(\pi^nm\right)^s\mathrm ds,
\end{align*}
where
\[\widetilde f(s)=\int\limits_0^\infty f(x)\,x^{s-1}\,\mathrm dx\]
is the Mellin transform of $f$, and $\sigma_0$ is a large positive real number, depending on the form.
\end{theorem}

We need to be able to simplify and understand the behaviour of the $\Gamma$-factors appearing in the above integrals, and for this purpose Stirling's formula is needed.
\begin{theorem}\label{stirling}
Fix some $K\in\mathbb Z_+$, $\delta\in\left]0,\pi\right[$ and $R\in\mathbb R_+$. Then
\[\Gamma(s)=\sqrt{2\pi}\,\exp\!\left(\Bigl(s-\frac12\Bigr)\log s-s\right)\left(1+\sum_{k=1}^K\frac{a_k}{s^k}+O_{K,\delta,R}\!\left(\left|s\right|^{-K-1}\right)\right)\]
for all $s\in\mathbb C$ with $\left|s\right|\geqslant R$ and $\left|\arg s\right|\leqslant\pi-\delta$. The constant coefficients $a_k$ do not depend on $K$, $\delta$ or $R$.

For vertical strips, $\left|\Gamma(s)\right|$ has the following useful estimate: given fixed real numbers $A<B$, we have for $s$ with $A\leqslant\sigma\leqslant B$ and $t\geqslant1$ that
\[\left|\Gamma(s)\right|=\sqrt{2\pi}\,t^{\sigma-1/2}\,e^{-\pi t/2}\left(1+O(t^{-1})\right).\]
\end{theorem}

We will need a lemma for estimating exponential integrals. The following is Lemma 6 in \cite{Jutila--Motohashi}.
\begin{lemma}\label{jutila-motohashi-lemma}
Let $a,b\in\mathbb R_+$ and $a<b$, and let $g\in C_{\mathrm c}^\infty(\mathbb R_+)$ with $\mathrm{supp}\,g\subseteq\left[a,b\right]$, and let $G_0$ and $G_1$ be such that
\[g^{(\nu)}(x)\ll_\nu G_0\,G_1^{-\nu}\]
for all $x\in\mathbb R_+$ for each nonnegative integer $\nu$. Also, let $f$ be a holomorphic function defined in $D\subseteq\mathbb C$, which consists of all points in the complex plane with distance smaller than $\rho\in\mathbb R_+$ from the interval $\left[a,b\right]$ of the real axis. Assume that $f$ is real-valued on $\left[a,b\right]$ and let $F_1\in\mathbb R_+$ be such that
\[F_1\ll\left|f'(z)\right|\]
for all $z\in D$. Then, for all positive integers $P$,
\[\int\limits_a^bg(x)\,e(f(x))\,\mathrm dx\ll_PG_0\left(G_1\,F_1\right)^{-P}\left(1+\frac{G_1}\rho\right)^P\left(b-a\right).\]
\end{lemma}

\section{Voronoi summation for Maass forms for $\mathrm{SL}(n,\mathbb Z)$}\label{asymptotics-section}

The goal is to derive asymptotics for the integrals appearing in Theorem \ref{gln-voronoi}. The relevant integral is
\[\Omega(y)=\frac1{2\pi i}\int\limits_{(-\sigma_0)}
\int\limits_0^\infty f(x)\,x^{s-1}\,\mathrm dx
\,\pi^{-n/2}\,\frac{\widetilde G(1-s)}{G(s)}\,y^s\,\mathrm ds,\]
where $f\in C_{\mathrm c}^\infty(\mathbb R_+)$.
The $x$-integral is certainly $\ll_N t^{-N}$, and so the $s$-integral behaves well.

\begin{theorem}\label{voronoi-asymptotics}
For any $K\in\mathbb Z_+$, $x\gg1$, $y\gg1$, we have
\[\Omega\!\left(y\right)
=y^{1/2+1/(2n)}
\int\limits_0^\infty f(x)\,x^{1/(2n)-1/2}\,\mathcal K(x,y)\,\mathrm dx,\]
where $\mathcal K$ has the asymptotics
\begin{multline*}
\mathcal K
=\sum_{\ell=0}^K
(xy)^{-\ell/n}\bigl(c_\ell^+\,e(\pi^{-1}\,n\,x^{1/n}\,y^{1/n})
+c_\ell^-\,e(-\pi^{-1}\,n\,x^{1/n}\,y^{1/n})\bigr)\\
+O\bigl((xy)^{-(K+1)/n}\bigr).
\end{multline*}
We emphasize that the implicit constant in the $O$-term is independent of $f$.
Here the leading coefficients $c_0^\pm$ are given by
\[c_0^\pm=\pi^{-(n+1)/2}\,\frac1{\sqrt n}\,e\!\left(\mp\frac{n+3}8\right).\]
\end{theorem}

\begin{proof}
The proof is based on the derivation of the asymptotics for integrals appearing in the ternary divisor function Voronoi summation formula in \cite{Ivic}, also performed for $\mathrm{GL}(3)$ Maass forms in \cite{Li}.

The first step in the proof is to replace the unwieldy $\Gamma$-factors by a simpler expression. In particular, we aim to replace them with a quotient of two $\Gamma$-factors.

We shall pick $\Lambda\in\mathbb R_+$ which is larger than the abcissa $\sigma_0$ in the original integral, larger in absolute values than any of the $\lambda_\ell$ and $\widetilde\lambda_\ell$, and larger than $K/n+1/(2n)-1/2$. 
Restricting to $s$ with $\left|s\right|\geqslant\Lambda$ and applying Stirling's formula, we get, after some simplification,
\begin{align*}
&\prod_{\ell=1}^n\frac{\Gamma\!\left(\frac{1-s-\widetilde\lambda_\ell(\nu)}2\right)}{\Gamma\!\left(\frac{s-\lambda_\ell(\nu)}2\right)}\\
&\qquad=\frac{\exp\!\left(n\frac{1-s-1}2\log\frac{1-s}2-n\frac{1-s}2+n\sum_{k=1}^K\frac{a_k'2^k}{(1-s)^k}\right)}{\exp\!\left(n\frac{s-1}2\log\frac s2-n\frac s2+n\sum_{k=1}^K\frac{a_k''2^k}{s^k}\right)}\left(1+O(\left|s\right|^{-K-1})\right).
\end{align*}
Here and below the implicit constants and the constant coefficients $a_k'$ and $a_k''$ will depend on $\lambda_\ell$ and $\widetilde\lambda_\ell$.

It is not hard to guess from this that if we want to replace these gamma factors by a quotient of two gamma functions, the new quotient should be
\begin{align*}
&\frac{\Gamma\!\left(\frac{1-ns}2\right)}{\Gamma\!\left(\frac{ns-(n-1)}2\right)}
\end{align*}

Some careful calculation reveals that the quotient of the old gamma factors and the new gamma factors is
\begin{align*}
&\prod_{\ell=1}^n\frac{\Gamma\!\left(\frac{1-s-\widetilde\lambda_\ell(\nu)}2\right)}{\Gamma\!\left(\frac{s-\lambda_\ell(\nu)}2\right)}
\cdot\frac{\Gamma\!\left(\frac{ns-(n-1)}2\right)}{\Gamma\!\left(\frac{1-ns}2\right)}
=
n^{ns-n/2}\left(1+\sum_{k=1}^K\frac{c_k}{s^k}+O\!\left(\left|s\right|^{-K-1}\right)\right).
\end{align*}
We write
\begin{align*}
H(s)&=n^{-ns+n/2}\prod_{\ell=1}^n\frac{\Gamma\!\left(\frac{1-s-\widetilde\lambda_\ell(\nu)}2\right)}{\Gamma\!\left(\frac{s-\lambda_\ell(\nu)}2\right)}\cdot\frac{\Gamma\!\left(\frac{ns-(n-1)}2\right)}{\Gamma\!\left(\frac{1-ns}2\right)}-1\\
&=\sum_{k=1}^K\frac{c_k}{s^k}+O\!\left(\left|s\right|^{-K-1}\right),
\end{align*}
and split the integral $\Omega(y)$ into two parts using $H(s)$:
\begin{align*}
\Omega(y)&=\frac1{2\pi i}\int\limits_{(-\sigma)}
\int\limits_0^\infty f(x)\,x^{s-1}\,\mathrm dx
\,\pi^{-n/2}\,n^{ns-n/2}\,\frac{\Gamma\!\left(\frac{1-ns}2\right)}{\Gamma\!\left(\frac{ns-(n-1)}2\right)}\,y^s\,\mathrm ds\\
&+\frac1{2\pi i}\int\limits_{(-\sigma)}
\int\limits_0^\infty f(x)\,x^{s-1}\,\mathrm dx
\,\pi^{-n/2}\,n^{ns-n/2}\,\frac{\Gamma\!\left(\frac{1-ns}2\right)}{\Gamma\!\left(\frac{ns-(n-1)}2\right)}\,H(s)\,y^s\,\mathrm ds.
\end{align*}

Using $\Lambda$, we shall write the Taylor expansion of $H(\cdot)$ at infinity as
\[H(s)=\sum_{k=1}^K\frac{c_k'}{\left(s+\Lambda\right)^{k}}+O\bigl(\left|s+\Lambda\right|^{-K-1}\bigr).\]
Now, from each term of the sum $\sum_{k=1}^K$ we get an integral
\[\frac1{2\pi i}\int\limits_{(-\sigma)}\int\limits_0^\infty f(x)\,x^{s-1}\,\mathrm dx\,\pi^{-n/2}\,n^{ns-n/2}\,\frac{\Gamma\!\left(\frac{1-ns}2\right)}{\Gamma\!\left(\frac{ns-(n-1)}2\right)}\,\left(s+\Lambda\right)^{-k}\,y^s\,\mathrm ds.\]
We get a similar integral from the $O$-term and it will be handled by Lemma \ref{the-final-O-term} below.

We shall consider, for $\nu,k\in\left\{0,1,2,\ldots,\right\}$, the more general integral
\begin{multline*}
\Omega_{\nu,k}(y)
=\frac1{2\pi i}\int\limits_{(-\sigma_0)}\int\limits_0^\infty f(x)\,x^{s-1}\,\mathrm dx\,\pi^{-n/2}\,n^{ns-n/2}\\
\cdot\frac{\Gamma\!\left(\frac{1-ns}2\right)}{\Gamma\!\left(\frac{ns+1}2+\nu-\frac n2\right)}\cdot\left(s+\Lambda\right)^{-k}\cdot y^s\,\mathrm ds.
\end{multline*}
Ultimately, all the main terms will come from $J$-Bessel functions which occur when $k=0$, and the error term comes from the case of large $k+\nu$ and the error term in the asymptotics of the $J$-Bessel functions.

The rest of the argument runs as follows:
\begin{enumerate}\setlength{\itemsep}{0pt}
\item We first compute asymptotics for $\Omega_{\nu,0}$ via residues. This, in particular, gives the main term which comes from $\Omega_{0,0}$.
\item We then consider the case of large $\nu+k$ by a simple shift of the line of integration and estimates by absolute values. The result will be small. Since the estimates are by absolute values, the argument also works with $(s+\Lambda)^{-k}$ replaced by $O((s+\Lambda)^{-k})$.
\item Finally, we will write $\Omega_{\nu,k}(y)$ in terms of $\Omega_{\nu+1,k-1}(y)$ and $\Omega_{\nu+1,k}(y)$.
Iterating this and combining it with the previous two steps gives the desired result by induction on $k$.
\end{enumerate}
These steps are taken in Lemmas \ref{the-case-k-0}, \ref{the-final-O-term} and \ref{induction-on-k} below.
\end{proof}

\begin{lemma}\label{the-case-k-0}
Let $\nu\in\left\{0,1,2,\ldots\right\}$. Then
\begin{multline*}
\Omega_{\nu,0}(y)
=-2\,\pi^{-n/2}\,y^{1/2+(1-\nu)/n}\,n^{-\nu}\\
\cdot\int\limits_0^\infty f(x)\,x^{(1-\nu)/n-1/2}\,J_{\nu-n/2}\bigl(2\,n\,x^{1/n}\,y^{1/n}\bigr)\,\mathrm dx.
\end{multline*}
\end{lemma}

\begin{proof}
The integrand has poles at $1/n$, $3/n$, $5/n$, \dots, apart possibly those finitely many points where the $\Gamma$-factor in the denominator has a pole, but we will not worry about this as formally everything will work out just fine (cf. e.g. Section 5.3 in \cite{Lebedev}).

The residue at $s=(2j+1)/n$, where $j\in\left\{0,1,2,\ldots\right\}$ is
\[\int\limits_0^\infty f(x)\,x^{(2j+1)/n-1}\,\mathrm dx\,\pi^{-n/2}\,n^{2j+1-n/2}\cdot\left(-\frac2n\right)\cdot\frac{(-1)^j}{j!}\cdot\frac{y^{(2j+1)/n}}{\Gamma\!\left(j+1+\nu-\frac n2\right)}.\]
The series of the residues is
\begin{align*}
&\sum_{j=0}^\infty\int\limits_0^\infty f(x)\,x^{(2j+1)/n-1}\,\mathrm dx\,\pi^{-n/2}\,n^{2j+1-n/2}\cdot\left(-\frac2n\right)\frac{(-1)^j}{j!}\cdot\frac{y^{(2j+1)/n}}{\Gamma\!\left(j+1+\nu-n/2\right)}\\
&=-2\,\pi^{-n/2}\int\limits_0^\infty f(x)\sum_{j=0}^\infty\frac{(-1)^j}{j!}\cdot\frac{\bigl(2n\,x^{1/n}\,y^{1/n}\bigr)^{2j+\nu-n/2}}{\Gamma\!\left(j+1+\nu-\frac n2\right)}\\
&\qquad\cdot2^{n/2-\nu-2j}\cdot n^{-\nu}\,x^{2j/n+1/n-1-2j/n-\nu/n+1/2}
\,y^{2j/n+1/n-2j/n-\nu/n+1/2}\,\mathrm dx\\
&=-2\,\pi^{-n/2}\,y^{1/2+(1-\nu)/n}\,n^{-\nu}\int\limits_0^\infty f(x)\,x^{(1-\nu)/n-1/2}\,J_{\nu-n/2}\bigl(2n\,x^{1/n}\,y^{1/n}\bigr)\,\mathrm dx.
\end{align*}
Here the exchange of the summation and integration is allowed because the power series of the $J$-Bessel function has the usual pleasant properties (cf. Section 5.3 in \cite{Lebedev}).

Thus, the result follows by shifting the line of integration in $\Omega_{\nu,0}(y)$ to $\sigma=2N/n$ for a large positive integer $N$, and letting $N\longrightarrow\infty$. First, the shift of the line of integration by a finite amount is justified by observing that in a vertical strip, the Mellin transform of $f$ is rapidly decaying and by Stirling's formula the quotient of $\Gamma$-factors is polynomially bounded, and when moving to right, the $\Gamma$-factors give more and more decay.

Thus, it only remains to prove that the integral over the line $\sigma=2N/n$ tends to zero. The $\Gamma$-factor of the numerator can be estimated as follows: We repeatedly use the functional equation $s\,\Gamma(s)=\Gamma(s+1)$ to shift the argument of the $\Gamma$-function to the line $\Re s=1/2$, and then use the latter part of Theorem \ref{stirling}. This leads to
\[\Gamma\!\left(\frac12-N-\frac{int}2\right)\ll\frac{e^{-nt/2}}{(N-1)!}.\]
Similarly the denominator may be estimated by
\[\frac1{\Gamma\!\left(N+\frac12+\nu-\frac n2+\frac{int}2\right)}
\ll\frac{e^{nt/2}}{\left\langle t\right\rangle^C},\]
where $C$ is a positive real constant. The Mellin transform of $f$ may be estimated by
\[\ll_{L,f}\left\langle t\right\rangle^{-L}\,B^N,\]
where $L$ is an arbitrary fixed positive integer, and $B$ is a large positive real constant which depends on the support of $f$.
Combining these gives the result as the exponential factors cancel each other and the factorial $(N-1)!$ grows faster than anything that remains.
\end{proof}

We next prove that when $\nu+k$ is large enough, the integrals $\Omega_{\nu,k}(y)$ are rather small.
\begin{lemma}\label{the-final-O-term}
For $\nu+k\geqslant K+2$, we have
\[\Omega_{\nu,k}(y)\ll
y^{1/2+1/(2n)}\int\limits_0^\infty\left|f(x)\right|x^{1/(2n)-1/2}\,\mathrm dx\,
O\bigl((xy)^{-(K+1)/n}\bigr),\]
with the implicit constant depending on $K$, $\Lambda$ and $n$. Furthermore, the same result holds, if the factor $\left(s+\Lambda\right)^{-k}$ in the definition of $\Omega_{\nu,k}(y)$ is replaced by any function which is analytic and $O\bigl((s+\Lambda)^{-k}\bigr)$ in the closed vertical strip bounded by the lines $\sigma=K/n+1/(2n)-1/2$ and $\sigma=-\sigma_0$.
\end{lemma}

\begin{proof}
We shift the line of integration to $\sigma=-X$, where
\[X=K/n+1/(2n)-1/2.\]
By Stirling's formula, the integrand is in the $t$-aspect
\[\ll\frac{\left\langle t\right\rangle^{(1+nX)/2}}{\left\langle t\right\rangle^{(nX+1)/2+\nu-n/2}}\cdot\left\langle t\right\rangle^{-k}
=\left\langle t\right\rangle^{nX+n/2-\nu-k}.\]
Therefore, since $nX=K+1/2-n/2\leqslant-n/2+\nu+k-3/2$, we may use Fubini's theorem to obtain
\begin{multline*}
\Omega_{\nu,k}(y)
=\pi^{-n/2}\,n^{-n/2}
\int\limits_0^\infty f(x)
\frac1{2\pi i}\int\limits_{(-X)}
x^{s-1}\,n^{ns}
\\
\cdot\frac{\Gamma\!\left(\frac{1-ns}2\right)}{\Gamma\!\left(\frac{ns+1}2+\nu-\frac n2\right)}\left(s+\Lambda\right)^{-k}\,y^s\,
\mathrm ds\,
\mathrm dx.
\end{multline*}
Now the claim follows from estimation by absolute values as the inner integral is
\begin{multline*}
\ll x^{-X-1}\,n^{-nX}\,y^{-X}\int\limits_{(-X)}\left|\frac{\Gamma(\ldots)}{\Gamma(\ldots)}\right|\cdot\left|s+\Lambda\right|^{-k}\mathrm ds\ll_X
x^{-X-1}\,y^{-X}\\
=x^{1/(2n)-1/2-(K+1)/n}\,y^{1/2+1/(2n)-(K+1)/n}.
\end{multline*}
\end{proof}

\begin{lemma}\label{induction-on-k}
Let $\nu,k\in\left\{0,1,2,\ldots\right\}$, and let $K$ and $\Lambda$ be as before, and assume that $K\geqslant\nu+k$. Then
\[\Omega_{\nu,k}(y)=2\,y^{1/2+1/(2n)}\int\limits_0^\infty f(x)\,x^{1/(2n)-1/2}\,\mathcal K(x,y)\,\mathrm dx,\]
where the kernel $\mathcal K(x,y)$ has the asymptotics
\begin{multline*}
\mathcal K(x,y)
=\sum_{\ell=\nu+k}^K(xy)^{-\ell/n}
\left(c_\ell^+\,e\bigl(\pi^{-1}\,n\,x^{1/n}\,y^{1/n}\bigr)
+c_\ell^-\,e\bigl(\pi^{-1}\,n\,x^{1/n}\,y^{1/n}\bigr)\right)\\
+O\bigl((xy)^{-(K+1)/n}\bigr).
\end{multline*}
The constants $c_\ell^\pm$, the kernel $\mathcal K$, and the implicit constant in the $O$-term depend on $n$, $\nu$, $k$, $K$ and $\Lambda$.
\end{lemma}

\begin{proof}
We shall perform induction on $k$. The case $k=0$ follows immediately from Lemma \ref{the-case-k-0} and the asymptotics for the $J$-Bessel function. Let us then assume that $k\in\mathbb Z_+$ and that the claim has already been proven for smaller values of $k$.

The key of the induction step lies in writing
\[\frac1{s+\Lambda}=\frac\alpha{\frac{ns+1}2+\nu-\frac n2}+\frac\beta{\left(\frac{ns+1}2+\nu-\frac n2\right)\left(s+\Lambda\right)}.\]
Since the right-hand side is
\[=\frac{\alpha s+\alpha\Lambda+\beta}{\left(\frac{ns+1}2+\nu-\frac n2\right)\left(s+\Lambda\right)},\]
this is achieved by choosing
\[\alpha=\frac n2\quad\text{and}\quad
\beta=\frac12+\nu-\frac n2-\frac{n\Lambda}2.\]
Thus, using the fact that $s\,\Gamma(s)=\Gamma(s+1)$, we obtain the recursion formula
\[\Omega_{\nu,k}(y)
=\alpha\,\Omega_{\nu+1,k-1}(y)+\beta\,\Omega_{\nu+1,k}(y).\]
Applying this formula to the last term and repeating, say, $N\in\mathbb Z_+$ times gives
\begin{multline*}
\Omega_{\nu,k}(y)
=\alpha_1\,\Omega_{\nu+1,k-1}(y)
+\alpha_2\,\Omega_{\nu+2,k-1}(y)
+\alpha_3\,\Omega_{\nu+3,k-1}(y)\\
+\ldots
+\alpha_N\,\Omega_{\nu+N,k-1}(y)
+\beta_N\,\Omega_{\nu+N,k}(y).
\end{multline*}
Here the last term is covered by Lemma \ref{the-final-O-term}, provided that $\nu+k+N\geqslant K+2$, and all the other terms are covered by the induction assumption.
\end{proof}

\section{Proofs of the resonance results}

Let us first prove Theorem \ref{resonance}.
 
\begin{proof}[Proof of Theorem \ref{resonance}]
\noindent Write $f(x)$ for $w(x)\,e(d^{1/n}\,x/M^{1-1/n})$, and let $\Omega(\cdot)$ be as in Section \ref{asymptotics-section}. Theorem \ref{gln-voronoi} says that
\begin{align}\label{sum1}
\sum_{M\leqslant m\leqslant M+\Delta}A(m,1,...,1)\,w(m)\,e\!\left(\frac{d^{1/n}\,m}{M^{1-1/n}}\right)=\sum_{m=1}^\infty\frac{A(1,...,1,m)}{m}\,\Omega(\pi^n\,m).
\end{align}
The integral $\Omega(\pi^n\,m)$ will be handled using the asymptotics given in Theorem \ref{voronoi-asymptotics}.
Let us first treat the error term arising from the kernel $\mathcal K$. For a sufficiently large $K\in\mathbb Z_+$, the contribution of the error term to the sum \eqref{sum1} equals
\begin{multline*}
\ll\sum_{m=1}^\infty\frac{\left|A(1,\ldots,1,m)\right|}{m^{1/2-1/(2n)}}\int\limits_M^{M+\Delta}\left|w(x)\right|x^{1/(2n)-1/2}\,(xm)^{-(K+1)/n}\,\mathrm dx\\
\ll_K\Delta\,M^{1/(2n)-1/2-(K+1)/n}.
\end{multline*} 
For $K$ fixed and sufficiently large depending on $n$, this is $\ll1$.
Now we move to the main term coming from $\mathcal K$. We are left with estimating integrals
\begin{align*}
\int\limits_M^{M+\Delta}g_m^\pm(x)\,w(x)\,e\!\left(\frac{d^{1/n}\,x}{M^{1-1/n}}\pm n\,x^{1/n}\,m^{1/n}\right)x^{1/(2n)-1/2}\,m^{1/2+1/(2n)}\,\mathrm dx,
\end{align*}
where
\begin{align*}
g_m^\pm(x)=\sum_{\ell=0}^Kc_\ell^\pm\,(xm)^{-\ell/n}.
\end{align*}

The derivative of the phase function, for $d\neq m$ in the case of the minus-sign and for all $m$ otherwise, is 
\begin{align*}
\frac{d^{1/n}}{M^{1-1/n}}\pm x^{1/n-1}m^{1/n}\asymp\frac{m^{1/n}}{M^{1-1/n}}.
\end{align*}
Therefore, using Lemma \ref{jutila-motohashi-lemma} we get for any $P\in\mathbb Z_+$ that
\begin{multline*}
\int\limits_M^{M+\Delta}g_m^\pm(x)\,w(x)\,e\!\left(\frac{d^{1/n}x}{M^{1-1/n}}\pm n\,x^{1/n}\,m^{1/n}\right)x^{1/(2n)-1/2}\,m^{1/2+1/(2n)}\,\mathrm dx\\
\ll M^{1/(2n)-1/2}\,m^{1/(2n)+1/2}\left(\Delta\cdot\frac{m^{1/n}}{M^{1-1/n}}\right)^{-P}\left(1+\frac{2\Delta}M\right)^P\cdot\Delta\\
\ll_P\Delta^{-P+1}\,m^{1/(2n)+1/2-P/n}\,M^{1/(2n)-1/2+((n-1)P)/n},
\end{multline*}
for these values of $m$.
The series of these integrals contribute, for a sufficiently large fixed $P$,
\[\ll_P\sum_{m=1}^\infty\frac{\left|A(1,\ldots,1,m)\right|}m
\cdot\Delta^{-P+1}\,m^{1/2n+1/2-P/n}\,M^{1/2n-1/2+(n-1)P/n}\ll_P1.\]

Substituting this back to the sum ($\ref{sum1}$) yields
\begin{align*}
&\sum_{M\leqslant m\leqslant M+\Delta}A(m,1,...,1)\,w(m)\,e\!\left(\frac{d^{1/n}\,m}{M^{1-1/n}}\right)=\frac{A(1,...,1,d)}{d^{1/2-1/(2n)}}\\
&\qquad\cdot\int\limits_M^{M+\Delta}g_d^-(x)\,e\!\left(\frac{d^{1/n}x}{M^{1-1/n}}-n\,x^{1/n}\,d^{1/n}\right)\,w(x)\,x^{1/(2n)-1/2}\,\mathrm dx
+O(1).
\end{align*}
The term with $\ell=0$ in the sum $g_d^{-}(x)$ will give the main term. Next we consider terms with $\ell\geqslant1$.
Integrating using absolute values gives
\begin{multline*}
\int\limits_M^{M+\Delta}c_\ell^-\,w(x)\,e\!\left(\frac{d^{1/n}\,x}{M^{1-1/n}}-n\,x^{1/n}\,d^{1/n}\right)x^{-\ell/n}\,x^{1/(2n)-1/2}\,\mathrm dx\\
\ll\Delta\,M^{-\ell/n+1/(2n)-1/2}
\ll\Delta\,M^{-1/2-1/(2n)}.
\end{multline*}

Hence
\begin{align*}
&\sum_{M\leqslant m\leqslant M+\Delta}A(m,1,...,1)\,w(m)\,e\!\left(\frac{d^{1/n}\,m}{M^{1-1/n}}\right)\\
&=\frac{A(1,...,1,d)}{d^{1/2-1/(2n)}}\,c_0^-\int\limits_M^{M+\Delta}w(x)\,e\!\left(\frac{d^{1/n}\,x}{M^{1-1/n}}-n\,x^{1/n}\,d^{1/n}\right)x^{1/(2n)-1/2}\,\mathrm dx\\
&\qquad+O(\Delta\,M^{-1/2-1/(2n)})+O(1)\\
&=\frac{A(1,...,1,d)}{d^{1/2-1/(2n)}\,\sqrt n}\,e\!\left(\frac{n+3}8\right)\\
&\qquad\cdot\int\limits_M^{M+\Delta}w(x)\,e\!\left(\frac{d^{1/n}\,x}{M^{1-1/n}}-n\,x^{1/n}\,d^{1/n}\right)x^{1/(2n)-1/2}\,\mathrm dx\\
&\qquad+O(\Delta\,M^{-1/2-1/(2n)}),
\end{align*}
which is what we wanted.
\end{proof}

\begin{proof}[Proof of Theorem \ref{nonlinear-resonance}]
This is similar to the previous proof. 
We first use the Voronoi type summation formula. The terms with $d\ne m$ are treated similarly as before, and in the end the term corresponding to the parameters $m=d$ and $\ell=0$ again yields the main term. Indeed, the only essential difference to the previous proof is that, in the current setting, the exponential factor completely disappears from the resonating integral.
\end{proof}

\section{$\Omega$-results}

From Corollary \ref{weighted-corollary}, we also get large values for unweighted short sums as follows.
\begin{corollary}\label{weightless-omega}
Let $d\in\mathbb Z_+$ be fixed and such that $A(1,\ldots,1,d)\neq0$. Then there exists $\Delta\asymp M^{1-1/(2n)}$ such that
\[
\sum_{M\leqslant m\leqslant M+\Delta}A(m,1,1,\dots ,1)\,e\!\left(\frac{d^{1/n}}{M^{1-1/n}}m\right)\gg M^{1/2}.
\]
\end{corollary}
\begin{proof}
Assume
\[\sum_{M\leqslant m\leqslant M+\Delta}A(m,1,1,\dots ,1)\,e\!\left(\frac{d^{1/n}}{M^{(n-1)/n}}m\right)=o(M^{1/2})\]
for all $\Delta\asymp M^{1-1/(2n)}$ and that $w(x)=1$ when $x\in [M_1,M_2]$. Further assume $M_1-M\asymp M_2-M_1\asymp M+\Delta -M_2$. Now
\begin{align*}
&\sum_{M\leqslant m\leqslant M+\Delta}A(m,1,1,\dots ,1)\,e\!\left(\frac{d^{1/n}}{M^{1-1/n}}m\right)w(m)\\
&=\left(\sum_{M\leqslant m\leqslant M_1}+\sum_{M_2\leqslant m\leqslant M+\Delta}\right)A(m,1,1,\dots ,1)\,e\!\left(\frac{d^{1/n}}{M^{1-1/n}}m\right)w(m)\\
&\qquad+\sum_{M_1\leqslant m\leqslant M_2}A(m,1,1,\dots ,1)\,e\!\left(\frac{d^{1/n}}{M^{1-1/n}}m\right)\\
&=\left(\sum_{M\leqslant m\leqslant M_1}+\sum_{M_2\leqslant m\leqslant M+\Delta}\right)A(m,1,1,\dots ,1)\,e\!\left(\frac{d^{1/n}}{M^{1-1/n}}m\right)w(m)\\
&\qquad+o(M^{1/2}).
\end{align*}
We may now use partial summation to bound the sums. We obtain
\begin{align*}
&\left(\sum_{M\leqslant m\leqslant M_1}+\sum_{M_2\leqslant m\leqslant M+\Delta}\right)A(m,1,1,\dots ,1)\,e\!\left(\frac{d^{1/n}}{M^{1-1/n}}m\right)w(m)\\
&=w(M_1)\sum_{M\leqslant m\leqslant M_1}A(m,1,1,\dots ,1)\,e\!\left(\frac{d^{1/n}}{M^{1-1/n}}m\right)\\
&\qquad+w(M_2)\sum_{M_2\leqslant M+\Delta}A(m,1,1,\dots ,1)\,e\!\left(\frac{d^{1/n}}{M^{1-1/n}}m\right)\\
&\qquad-\int\limits_{M}^{M_1}\left(\sum_{M\leqslant m\leqslant t}A(m,1,1,\dots ,1)\,e\!\left(\frac{d^{1/n}}{M^{1-1/n}}m\right) \right)w'(t)\,\mathrm dt\\
&\qquad-\int\limits_{M_2}^{M+\Delta}\left(\sum_{M\leqslant m\leqslant t}A(m,1,1,\dots ,1)\,e\!\left(\frac{d^{1/n}}{M^{1-1/n}}m\right) \right)w'(t)\,\mathrm dt\\
&=o(M^{1/2})+o\left(\int\limits_M^{M_1}M^{1/2}\,\Delta^{-1}\,dt\right)+o\left(\int\limits_{M_2}^{M+\Delta}M^{1/2}\,\Delta^{-1}\,dt\right)=o(M^{1/2}),
\end{align*}
which leads to a contradiction with Corollary \ref{weighted-corollary}. This proves the corollary.
\end{proof}

Since a sum can be estimated by splitting it into shorter sums and then adding up upper bounds for the subsums, we get the $\Omega$-result also for shorter sums:
\begin{corollary}\label{shorter-weightless-omega} Let $d\in\mathbb Z_+$ be fixed and such that $A(1,\ldots,1,d)\neq0$. Then, for any $\gamma\in\left[1/2-1/(2n),1-1/(2n)\right]$, there exists $\Delta\asymp M^\gamma$ such that
\[
\sum_{M\leqslant m\leqslant M+\Delta}A(m,1,1,\dots ,1)\,e\!\left(\frac{d^{1/n}}{M^{1-1/n}}\,m\right)\gg\Delta\,M^{1/(2n)-1/2}.
\]
\end{corollary}

Let us state a lemma about the size of exponential sums before proving Theorem \ref{omega}.
\begin{lemma}
We have
\[
\sum_{0\leqslant h\leqslant\Delta-1}e\!\left(\frac{d^{1/n}}{M^{1-1/n}}\,h\right)\asymp\Delta
\]
when $\Delta=o\left(M^{(n-1)/n}\right)$ and $d\in\mathbb Z_+$ is fixed.
\end{lemma}
\begin{proof}
Since $\sum_{0\leqslant h\leqslant\Delta-1}e\!\left(\frac{d^{1/n}}{M^{1-1/n}}\,h\right)$ is a geometric sum, we have
\[
\sum_{0\leqslant h\leqslant\Delta-1}e\!\left(\frac{d^{1/n}}{M^{1-1/n}}\,h\right)=\frac{e\!\left(\frac{d^{1/n}}{M^{(n-1)/n}}\,\Delta\right)-1}{e\!\left(\frac{d^{1/n}}{M^{(n-1)/n}}\right)-1}\asymp \frac{d^{1/n}\,M^{-(n-1)/n}\,\Delta}{d^{1/n}\,M^{-(n-1)/n}}=\Delta.
\]
\end{proof}

We are now ready to move to the proof of Theorem \ref{omega}.

\begin{proof}[Proof of Theorem \ref{omega}]
Let $U \asymp M^{1-1/(2n)}$ be such that
\[
\sum_{M\leqslant m\leqslant M+U }A(m,1,1,\dots ,1)\,e\!\left(\frac{d^{1/n}}{M^{(n-1)/n}}m\right)\gg M^{1/2}.
\]
Now
\begin{multline*}
\left(\sum_{0\leqslant h\leqslant \Delta-1}e\!\left(\frac{d^{1/n}}{M^{1-1/n}}h\right)\right)\left(\sum_{M\leqslant m\leqslant M+U }A(m,1,1,\dots ,1)\,e\!\left(\frac{d^{1/n}}{M^{(n-1)/n}}m\right)\right)\\
\gg M^{1/2}\,\Delta.
\end{multline*}
On the other hand, we have
\begin{align*}
&\left(\sum_{0\leqslant h\leqslant \Delta-1}e\!\left(\frac{d^{1/n}}{M^{1-1/n}}h\right)\right)\left(\sum_{M\leqslant m\leqslant M+U }A(m,1,1,\dots ,1)\,e\!\left(\frac{d^{1/n}}{M^{(n-1)/n}}m\right)\right)\\
&=\sum_{\substack{0\leqslant h\leqslant \Delta-1\\M\leqslant m\leqslant M+U }}A(m,1,1,\dots ,1)\,e\!\left(\frac{d^{1/n}}{M^{1-1/n}}(m+h)\right)\\
&=e\!\left(\frac{d^{1/n}}{M^{1-1/n}}\right)A(m,1,\dots ,1)\\
&\qquad+e\!\left(\frac{d^{1/n}}{M^{1-1/n}}(M+1)\right)\left(A(m,1,\dots,1)+A(m+1,1,\dots,1)\right)+\ldots \\
&\qquad+e\!\left(\frac{d^{1/n}}{M^{1-1/n}}(M+\Delta-1)\right)\sum_{M\leqslant m\leqslant M+\Delta-1}A(m,1,\dots,1)\\
&\qquad\qquad\ldots\\
&\qquad+e\!\left(\frac{d^{1/n}}{M^{1-1/n}}(M+U )\right)\sum_{-\Delta+1+M+U \leqslant m\leqslant M+U }A(m,1,\dots ,1)\\
&\qquad+\ldots+e\!\left(\frac{d^{1/n}}{M^{1-1/n}}(M+U +\Delta-1)\right)A(M+U +\Delta-1,1,\dots,1).
\end{align*}
The sums that are of length $<\Delta$ are at most of size $\Delta^{1+\varepsilon+\vartheta}$.
Hence, their total contribution is $\ll\Delta^{2+\varepsilon+\vartheta}$,
so that when $\Delta^{2+\varepsilon+\vartheta}=o\left(M^{1/2}\,\Delta\right)$, we have
\begin{multline*}
e\!\left(\frac{d^{1/n}}{M^{1-1/n}}(M+\Delta-1)\right)\sum_{M\leqslant m\leqslant M+\Delta-1}A(m,1,\dots,1)+\ldots \\
+e\!\left(\frac{d^{1/n}}{M^{1-1/n}}(M+U )\right)\sum_{-\Delta+1+M+U \leqslant m\leqslant M+U }A(m,1,\dots ,1)\gg M^{1/2}\,\Delta.
\end{multline*}
Since the number of sums is $U $, there must be a sum of size
\[\gg M^{1/2}\,\Delta\,U ^{-1}
\asymp M^{1/2}\,\Delta\,M^{1/(2n)-1}
=\Delta\,M^{1/(2n)-1/2}.\]
\end{proof}

\section*{Funding}

This work was supported by the Academy of Finland [grant no. 138337 to A.-M. E.-H., grant no. 138522 to J. J., Finnish Centre of Excellence in Inverse Problems Research to E. V.], Finland's Ministry of Education [Doctoral Program in Inverse Problems to E. V.], and the Vilho, Yrj\"o and Kalle V\"ais\"al\"a Foundation [to E. V.].

\end{document}